\documentclass[a4paper, 12pt]{article} 

\usepackage[utf8]{inputenc}
\usepackage[T1]{fontenc}
\usepackage[a4paper, margin=2.75cm]{geometry}
\usepackage{needspace, mathscinet}
\usepackage{blkarray} 
\usepackage{multirow}
\usepackage{libertine} % text font
\usepackage{inconsolata} % texttt font
\usepackage{nicefrac}

\usepackage{authblk, amsmath, amsthm, amssymb, graphicx, thmtools, thm-restate, mathtools}

\usepackage[textsize=small, textwidth=2cm, color=yellow]{todonotes}
\usepackage[colorlinks=true, citecolor=blue, linkcolor=blue, urlcolor=blue]{hyperref}
\usepackage{enumitem}
\usepackage[noabbrev,capitalize,nameinlink]{cleveref}
\crefname{claim}{Claim}{Claims}
\crefname{lemma}{Lemma}{Lemmas}
\crefname{theorem}{Theorem}{Theorems}
\crefname{proposition}{Proposition}{Propositions}
\crefname{definition}{Definition}{Definitions}
\crefname{conjecture}{Conjecture}{Conjectures}
\crefname{corr}{Corollary}{Corollaries}

\declaretheorem[name=Theorem, numberwithin=section]{theorem}
\declaretheorem[name=Lemma, sibling=theorem]{lemma}

\declaretheorem[name=Problem, sibling=theorem]{problem}

\declaretheorem[name=Corollary, sibling=theorem]{corollary}

\declaretheorem[name=Remark, sibling=theorem]{remark}
\declaretheorem[name=Claim]{claim}

\def\cqedsymbol{\ifmmode$\lrcorner$\else{\unskip\nobreak\hfil
\penalty50\hskip1em\null\nobreak\hfil$\lrcorner$
\parfillskip=0pt\finalhyphendemerits=0\endgraf}\fi}

\newenvironment{claimproof}[1][\proofname]
{%
	\proof[#1]%
}
{%
	\endproof%
}

\let\leq\leqslant
\let\geq\geqslant

\let\OLDthebibliography\thebibliography
\renewcommand\thebibliography[1]{
  \OLDthebibliography{#1}
  \setlength{\parskip}{0pt}
  \setlength{\itemsep}{0pt plus 0.3ex}
}%
\AtBeginDocument{%Remove MR numbers from bibliography
   \def\MR#1{}
}%

\newcommand{\maxCircle}{98}

\newcommand{\chio}{\chi_{\mathrm{o}}}
\newcommand{\chiio}{\chi_{\mathrm{io}}}
\newcommand{\chiso}{\chi_{\mathrm{so}}}
\newcommand{\chiiso}{\chi_{\mathrm{iso}}}
\newcommand{\chipcf}{\chi_{\mathrm{pcf}}}
\newcommand{\chicf}{\chi_{\mathrm{cf}}}

\newcommand{\GF}{\textsf{GF}}

\title{Odd coloring graphs with linear neighborhood complexity}

\author{James Davies}
\affil{Faculty of Mathematics and Computer Science, Leipzig University, Germany.}
\author{Meike Hatzel\thanks{Partly supported by the Institute for Basic Science (IBS-R029-C1).}}
\affil{Department of Mathematics, Technical University of Darmstadt, Germany.}
\author{Kolja Knauer\thanks{Supported through grant 
 PID2022-137283NB-C22 funded by MICIU/AEI/10.13039/501100011033 and by ERDF/EU and through the Severo Ochoa and María de Maeztu Program for Centers and Units of Excellence in R\&D (CEX2020-001084-M).}}
\affil{Departament de Matem\'{a}tiques i Inform\'{a}tica, Universitat de Barcelona, Centre de Recerca Matemàtica, Barcelona, Spain.}
\author{Rose McCarty\thanks{Supported by the National Science Foundation under Grant No.~DMS-2202961 and~DMS-2452111.}}
\affil{School of Mathematics and School of Computer Science, Georgia Institute of Technology, USA.}
\author{Torsten Ueckerdt\thanks{Supported by the Deutsche Forschungsgemeinschaft (DFG, German Research Foundation) -- 520723789.}}
\affil{Institute of Theoretical Informatics, Karlsruhe Institute of Technology, Germany.}
\begin{document}

\maketitle

\begin{abstract}
    We prove that any class of graphs with linear neighborhood complexity has bounded improper odd chromatic number. As a result, if $\mathcal{G}$ is the class of all circle graphs, or if $\mathcal{G}$ is any class with bounded twin-width, bounded merge-width, or a forbidden vertex-minor, then $\mathcal{G}$ is $\chio$-bounded.
    %Our proofs rely on a previously unexplored connection with Haussler's Shallow Packing Lemma.
    %\RMM{I updated the first sentence.}
\end{abstract}

%INTRODUCTION
\section{Introduction}

The \emph{odd chromatic number} $\chio(G)$ of a graph $G$ is the minimum number of colors needed to properly color $G$ so that for each non-isolated vertex $v$ of $G$, there exists a color which appears on an odd number of vertices in its open neighborhood $N(v)$. This parameter\footnote{Note that in the literature there exists another parameter called \emph{odd chromatic number}, see e.g.~\cite{BHKM23}.} was introduced by Petru{\v{s}}evski and {\v{S}}krekovski~\cite{PS22}. Among the main problems and results concerning this parameter is the conjecture of~\cite{PS22} that all planar graphs have odd chromatic number at most $5$, where $8$ is the best current upper bound~\cite{PP23}. Related classes that have been studied are $1$-planar graphs~\cite{LWY23}, IC-planar graphs~\cite{PWL23}, toroidal graphs~\cite{TY23}, graphs of bounded thickness~\cite{Kit24}, and graphs of bounded maximum average degree~\cite{CCKP23,Cra24,WY24}. Another interesting conjecture in the area, due to Caro, Petru{\v{s}}evski and {\v{S}}krekovski~\cite{CPS22}, states that the odd chromatic number of a graph is at most its maximum degree plus a constant. The current best results can be found in~\cite{DOP24}. 

While Cranston~\cite{Cra24} constructed graphs with maximum average degree at most 4 and unbounded odd chromatic number, there are other more restrictive notions of sparsity that do bound the odd chromatic number. In particular, classes of bounded expansion have bounded odd chromatic number; indeed, they even have bounded proper conflict free-chromatic number~\cite{Hic23,Liu24} and bounded strong odd chromatic number~\cite{Pil25}. See the end of this section for the relevant definitions.

A convenient relaxed notion is the \emph{improper odd chromatic number} $\chiio(G)$ of a graph $G$. It is the minimum number of colors needed to (possibly improperly) color the vertices of $G$ so that for each non-isolated vertex $v$ of $G$, there exists a color which appears on an odd number of vertices in its open neighborhood $N(v)$. Since a product coloring of an improper odd coloring and a proper coloring yields a proper odd coloring, we obtain the following. We denote the chromatic number of a graph $G$ by $\chi(G)$.

\begin{remark}\label{rem:product}
    For every graph $G$ we have $\chio(G)\leq\chi(G)\chiio(G)$.    
\end{remark}

In the current paper, we prove that any graph of linear neighborhood complexity has bounded improper odd chromatic number, which then yields bounds on the odd chromatic number once the chromatic number is bounded. There are many examples of graphs with linear neighborhood complexity. In particular, classes of graphs with good model-theoretic properties often have linear neighborhood complexity; see, for example, \cite{mergewidthChi, complexityTWW, flipperGamesMonStable, RVS19}. The \emph{neighborhood complexity} of a graph $G$, denoted by $\eta_G(m)$, is the maximum over all $m$-element subsets $A \subseteq V(G)$ of $|\{N(v) \cap A: v \in V(G)\}|$. Thus, the neighborhood complexity counts the number of different neighborhoods that can occur within an $m$-vertex set. 

\begin{restatable}{theorem}{nghComplex}
\label{thm:main}
    For any integer $r$, there exists an integer $k = k(r)$ so that if $G$ is a graph with neighborhood complexity $\eta_G(m) \leq r \cdot m$ for every positive $m$, then $G$ can be (improperly) $k$-colored such that for each non-isolated vertex $v$, there exists a color which appears an odd number of times in its neighborhood $N(v)$.
\end{restatable}
\noindent The main ingredient in our proof of \cref{thm:main} is Haussler's Shallow Packing Lemma~\cite{Haussler95}.

We obtain the following corollary of \cref{thm:main} using \cref{rem:product} and, in the case of vertex-minors, a technique from~\cite{jiménez2024boundednessproperconflictfreeodd}. 
%technique from~\cite{jiménez2024boundednessproperconflictfreeodd} to reduce to the case of bipartite graphs. 
A graph class is \emph{$\chio$-bounded}\footnote{Note that in~\cite{jiménez2024boundednessproperconflictfreeodd} this was defined via the existence of $f$ such that  $\chio(G)\leq f(\chi(G))$ for all $G$ in the class. Since the classes we consider are $\chi$-bounded, both definitions coincide for the results in the present paper.} if there exists a function $f$ so that any graph in the class with clique number $\omega$ has odd chromatic number at most~$f(\omega)$. 

\begin{restatable}{corollary}{corWidthEtc}
\label{cor:NeighborhoodComplOddColoring}
    If $\mathcal{G}$ is any class of graphs with bounded merge-width or with a forbidden vertex-minor, then $\mathcal{G}$ is $\chio$-bounded. 
\end{restatable}

\noindent We note that every class of bounded twin-width also has bounded merge-width~\cite[Theorem~1.4]{mergeWidthIntro}. So Corollary~\ref{cor:NeighborhoodComplOddColoring} implies that classes of bounded twin-width are $\chio$-bounded. The chromatic number of graphs with bounded twin-width has been well-studied; see~\cite{twIII, Bourneuf2025Bounded, PilipczukSokolowski23}. Also see~\cite{twinWidthIntro} for the original motivation of the parameter.

Our proof of \cref{cor:NeighborhoodComplOddColoring} uses the fact that classes with bounded merge-width~\cite{mergewidthChi} or a forbidden vertex-minor~\cite{Davies2022vm} are $\chi$-bounded. (A class of graphs is \emph{$\chi$-bounded} if there exists a function $f$ so that any graph $G$ in the class with clique number $\omega$ has $\chi(G)\leq f(\omega)$.)

We also use a connection between the Matroid Growth Rate Theorem~\cite{geelen2003cliques} and neighborhood complexity in order to obtain the following corollary of \cref{thm:main} for a class of bipartite graphs. This connection was recently observed by the authors~\cite{davies2025girthgfqrepresentablematroids}, who used it in order to find the unavoidable cosimple $\GF(q)$-representable matroids of large girth. %\RMM{Add citation.}\MH{added the citation}

\begin{corollary}
\label{cor:oddColoringMat}
    For any integers $t$ and $q$, there exists an integer $f(t,q)$ such that for any matroid $M$ without $U_{2,q+2}$, $U_{q,q+2}$, $M(K_t)$, or $M(K_t)^*$ as a minor and for any basis $B$ of $M$, the fundamental graph $\mathcal{F}(M,B)$ has odd chromatic number at most $f(t,q)$.
\end{corollary}

An interesting common special case of \cref{cor:NeighborhoodComplOddColoring} and \Cref{cor:oddColoringMat} is that of bipartite circle graphs. A graph is a \emph{circle graph} if there exists a collection of chords on a circle so that the vertices are in bijection with the chords, and two vertices are adjacent if and only if their chords have non-empty intersection. Bouchet~\cite{circleGraphsVM} 
characterized circle graphs by their forbidden vertex-minors. Additionally, circle graphs with bounded clique number have bounded merge-width~\cite{mergeWidthIntro, circleGraphsTwinWidth}. We also note that there are several direct proofs that circle graphs are $\chi$-bounded; see~\cite{improvedCircleColoring, circlePolyChi, circleMultInterval, polygonalCircle}. Moreover, bipartite circle graphs are the fundamental graphs of planar graphs~\cite{deFraysseix81}. 
Using a technique from~\cite{jiménez2024boundednessproperconflictfreeodd}, we apply our results about bipartite circle graphs to prove that all circle graphs are $\chio$-bounded. This settles~\cite[Question 3.9]{jiménez2024boundednessproperconflictfreeodd}, which asked whether circle graphs are $\chio$-bounded. 

Based on the fact that bipartite circle graphs are the fundamental graphs of planar graphs, we also develop another strategy to prove  $\chio$-boundedness. In the case of circle graphs, this strategy gives a much better bound.

For this result, recall that the \emph{star-arboricity} of a graph $G$ is the smallest integer $\ell$ such that $E(G)$ can be decomposed into $\ell$ star forests (that is, forests where each component has diameter at most $2$). %Note that the $\mathrm{sa}(G)$ is bounded by twice the degeneracy of $G$, see~\cite{AA89}. Hence, in particular 
By the Kostochka–Thomason bound~\cite{Kos82,Kos86,Tho84, Tho01}, for any integer $t$, every graph with no $K_t$ minor has star-arboricity $O(t\sqrt{\log t})$. For each integer $t$, we write $\ell(t)$ for the largest star-arboricity of any graph with no $K_t$-minor. Given a graph $G$, a spanning forest $T$ of $G$, and an edge $f \in E(T)$, note that deleting $f$ from $T$ yields exactly two components. The \emph{fundamental cut} of $T$ with respect to $f$ is the set of all edges in $E(G)$ which join these two components.
We can then prove the following ``one-sided'' coloring result.

%\RMM{I think this would be better to write without defining $\ell$, but I don't want to mess with changing it, and I guess for circle graphs maybe you want more precise bounds. I would also not define the notation $C^*$ yet (or at all if possible...). Can someone fix the definition of $C^*(T,f)$ above? I guess you want to include $f$, which is weird to me.}\KK{I included $\ell$ exactly because we needed this theorem to be super precise, (actually after a comment of Meike), when using it on circle graphs. Don't know how to avoid this.}\KK{I moved the notion for fundamental cuts into the proof pof the Theorem, where it is relatively convenient}\RMM{Got it. Looks good, thanks!}

\begin{restatable}{theorem}{thmfundCuts}
\label{thm:fundCuts}
    For any integer $t$, there exists an integer $k \leq 16\ell(t-1)+1$ so that if $G$ is a graph with no $K_t$ minor and $T$ is a spanning forest of $G$, then there exists a coloring $\phi: E(G)\setminus E(T) \rightarrow \{1,\ldots,k\}$ so that for each non-bridge $f \in E(T)$, there exists a color $i \in \{1,\ldots,k\}$ which occurs an odd number of times on the fundamental cut of $T$ with respect to $f$.%$C^*(T,f)$. %Indeed, $k(t)$, with $\ell(\cdot)$ defined as above the theorem. 
\end{restatable}

\cref{thm:fundCuts} can be seen as a one-sided coloring result in the fundamental graph of the graphic matroid of $G$, which is a matroid $M$ without $U_{2,4}$, $F_7$, $F^*_7$, $M^*(K_5)$, $M^*(K_{3,3})$, or $M(K_t)$ as a minor, where  $F_7$ is the \emph{Fano matroid} and $K_{3,3}$ is the complete bipartite graph with sides of size $3$. We believe that Theorem~\ref{thm:fundCuts} is interesting on its own and propose a possible strengthening; see \cref{prob:primal}. 
As an application of \cref{thm:fundCuts}, we further improve the bound on the odd chromatic number of bipartite circle graphs to~{\maxCircle}. This is a more explicit bound, even if it is far from the lower bound; see~\cref{prob:circle}. 

\begin{restatable}{corollary}{colBipCircle}
\label{cor:coloringFundPlanar}
Every bipartite circle graph has odd chromatic number at most \maxCircle{}.
\end{restatable}

Recall that our proof that circle graphs are $\chio$-bounded reduces to the case of bipartite circle graphs using the technique from~\cite{jiménez2024boundednessproperconflictfreeodd}. We note that this type of approach is necessary in the following sense.

\begin{restatable}{proposition}{propImpro}
\label{prop:Impro}
There are circle graphs of arbitrarily high improper odd chromatic number.
\end{restatable}

%\noindent We obtain this result as a corollary of the more general \cref{thm:fundCuts}, which can be seen as a ``one-sided'' odd-coloring result for the fundamental graph of a graphic matroid. We believe that \cref{thm:fundCuts} is interesting on its own and propose a possible strengthening; see \cref{prob:primal}. 

%(Bipartite circle graphs are exactly the fundamental graphs of planar graphs~\cite{deFraysseix81}.)

\subsection*{Related results and parameters} We now give a brief survey of how various versions of the chromatic number are related. 

\paragraph{Proper coloring parameters.}
The odd chromatic number was introduced as a relaxation of the \emph{proper conflict-free chromatic number} $\chipcf(G)$. Here for each non-isolated vertex $v$ of $G$, there must exist a color which appears on exactly one vertex in its open neighborhood $N(v)$. This parameter was introduced by Fabrici, Lu\v{z}ar, Rindo\v{s}ov\'a, and Sot\'ak~\cite{FLRS23}.
Both parameters can be strengthened by requiring that \emph{every} color in the open neighborhood of $v$ appears an odd number of times or exactly once, respectively. This leads to the \emph{strong odd chromatic number} $\chiso(G)$ introduced in~\cite{KP24}, and the \emph{squared chromatic numer}, i.e. the chromatic number $\chi(G^2)$ of $G^2$ the \emph{square} of $G$, where different vertices of $G^2$ are connected if their distance in $G$ is at most $2$, see~\cite{Cra23} for a recent account. If $\Delta(G)$ denotes the maximum degree of $G$, then from the above definitions, we immediately get the following inequalities.\begin{equation}
    \omega(G)\leq \chi(G)\leq \chio(G)\leq \chiso(G), \chipcf(G)\leq \chi(G^2)\leq \Delta(G)^2+1.
\end{equation}\label{prop}
The rightmost inequality in \eqref{prop} can also be reversed in some sense, since we trivially have $\chi(G^{2})\geq \Delta(G)+1$ for all graphs $G$.
However, all remaining inequalities in \eqref{prop} can be arbitrarily far apart:
Petru\v{s}evski and \v{S}krekovski~\cite{PS22} show that there is no function $f$ such that $\chio(G)\leq f(\chi(G))$ for all $G$.
It is also known~\cite{jiménez2024boundednessproperconflictfreeodd} that there is no function $f$ such that $\chipcf(G)\leq f(\chio(G))$ for all $G$.
Further, since $\Delta(G) < \chi(G^2)$, any graph class with unbounded degree but bounded $\chipcf$ shows that there is no function $f$ such that $\chi(G^2)\leq f(\chipcf(G))$ for all $G$, e.g. complete bipartite graphs. Kwon and Park~\cite{KP24} note that there is no function $f$ such that $\chi(G^2)\leq f(\chiso(G))$ for all $G$.
Also, there is no function $f$ such that $\chiso(G)\leq f(\chio(G))$ for all graphs~$G$, see~\cite{soon}. Just for completeness, let $G$ be obtained from $K_n$ by subdividing each edge once and then attaching one vertex incident to all others. Then $\chipcf(G)\leq 4$ but $\chiso(G)\geq n$. On the other hand, if $G$ is the bipartite graph with bipartition $(X,Y)$ where ${{|X|-1}\choose 2}$ is odd and for every $A\subseteq X$ with $|A|=3$ there is exactly one $v\in Y$ with $N(v)=A$, then $\chiso(G)=2$ but $\chipcf(G)\geq\frac{|X|}{2}$. So $\chiso$ and $\chipcf$ are not tied together by any function. 

Hence, one can wonder if for a class of graphs there is a function that upper bounds one of the above coloring parameters in terms of the clique number, yielding various notions of \emph{boundedness}. As mentioned in the beginning of the introduction, classes of bounded expansion are $\chipcf$-bounded~\cite{Hic23,Liu24} as well as $\chiso$-bounded~\cite{Pil25}. 

To see the relevance of our results, note that circle graphs are neither $\chipcf$-bounded~\cite{jiménez2024boundednessproperconflictfreeodd} nor $\chiso$-bounded~\cite{soon}. Hence they highlight the difference of $\chio$ when compared to $\chipcf$ and $\chiso$. Further, we can see that also the classes we consider behave different than other closely related classes. Namely, while the class of permutation graphs, a subclass of circle graphs, is $\chipcf$-bounded, several other superclasses of permutation graphs such as convex bipartite graphs, grid-intersection graph, and comparability graphs of 3-dimensional posets are not $\chio$-bounded, see~\cite{jiménez2024boundednessproperconflictfreeodd}.

\paragraph{Improper colorings.}
Each of the above coloring parameters can be relaxed by dropping the requirement of the coloring being proper. The proper conflict-free colorings were actually motivated by (improper) \emph{conflict-free chromatic number} $\chicf$ introduced by Pach and Tardos~\cite{PT09}. This came naturally from conflict-free colorings of hypergraphs studied by Even, Lotker, Ron, and Smorodinsky~\cite{ELRS03} in relation to the frequency assignment problem in cellular networks. The latter were later relaxed by Cheilaris, Keszegh, and Pálvölgyi to odd colorings for hypergraphs in~\cite{CKP13}. From this point of view the {improper odd chromatic number} $\chiio(G)$ arises naturally and has been considered in~\cite{CZ24, ZML25}. The {improper strong odd chromatic number} $\chiiso(G)$ was introduced in~\cite{soon}. The improper version of the squared chromatic number  is the chromatic number $\chi(G^{\sharp 2})$ of the \emph{exact square} $G^{\sharp 2}$ of $G$, where vertices of $G^{\sharp 2}$ are connected if their distance in $G$ is $2$, see~\cite{FHMNNSV21}. The improper variant of a parameter is always a lower bound for the proper variant. As mentioned in \cref{rem:product} a product coloring yields that $\chio(G)\leq\chiio(G)\chi(G)$. Similarly one gets $\chipcf(G)\leq\chicf(G)\chi(G)$ and $\chi(G^2)\leq \chi(G^{\sharp 2})\chi(G)$. However,  there exists no function $f$ such that that $\chiso(G)\leq f(\chiiso(G),\chi(G))$ for all $G$, see~\cite{soon}. %In the proof of \cref{thm:main} we will use that if 

\subsection*{Outline}
In \cref{sec:prelim}, we introduce Haussler's Shallow Packing Lemma and some preliminaries about matroids. We prove \cref{thm:main} in \cref{sec:oddChiFund}. In \cref{sec:applications}, we prove the two corollaries of \cref{thm:main} (\cref{cor:NeighborhoodComplOddColoring} and \cref{cor:oddColoringMat}). Finally, in \cref{sec:bipartiteCircle}, we use a different technique (\cref{thm:fundCuts}) to find a concrete bound on the odd chromatic number of bipartite circle graphs (\cref{cor:coloringFundPlanar}). We also prove Proposition~\ref{prop:Impro} and propose a couple of open problems.

%PRELIMINARIES
\section{Preliminaries}
\label{sec:prelim}
We introduce some basic notions of matroids and graphs, see~\cite{Oxl11,Wel76} and~\cite{Die25} for books on these topics.

A \emph{matroid} is a pair $M=(E,\mathcal{B})$ of a finite \emph{ground set} $E$ and a collection of bases $\mathcal{B}$. These are subsets of $E$ such that for all $B, B'\in \mathcal{B}$ and $e\in B\setminus B'$ there exists a $e'\in B'\setminus B$ such that $B\setminus\{e\}\cup\{e'\}\in\mathcal{B}$. The \emph{dual matroid}  of $M$ is $M^*=(E,\{E\setminus B\mid B\in \mathcal{B}\})$. A subset $C\subseteq E$ is a \emph{circuit} of $M$ is it is contained in no basis of $M$, and it is a \emph{cocircuit} of $M$ is it is a circuit of $M^*$. The \emph{girth} and \emph{cogirth} of a matroid are the size of its smallest circuit and cocircuit, respectively. If the girth and corgirth is larger than two, we say that $M$ is \emph{simple} and \emph{cosimple}, respectively. For an element $e \in E$, we write $M\setminus e$ for the matroid obtained from $M$ by deleting $e$, i.e., its bases are the largest independent sets of $M$ in $E\setminus \{e\}$. Furthermore, denote by $M/e$ the matroid obtained from $M$ by contracting $e$, i.e, $(M^*\setminus e)^*$. If $M'$ can be obtained from $M$ by a sequence of contractions and deletions, we say that $M'$ is a \emph{minor} of $M$. We write $U_{2,q}$ for the $q$-element line, and, more generally, $U_{t,q}$ for the uniform matroid with $q$ elements and rank $t$. That is, $U_{t,q}$ is the $q$-element matroid where the circuits are the sets of size $t+1$. We say that $M$ is \emph{$\GF(q)$-representable} if there exists a $\GF(q)$-vector space whose bases correspond to $M$. IF $q=2$, then $M$ is \emph{binary}.

Given a graph $G$, the bases of the \emph{graphic matroid} $M(G)$ of $G$ are the edge sets of spanning forests; a \emph{spanning forest} is a maximal acyclic subgraph of $G$.
Given a graph $G$ and a vertex $v$, we write $N_G(v)$ (or just $N(v)$ when the graph is clear from the context) for the \emph{open} neighborhood of $v$. That is, we do not consider $v$ to be in its neighborhood $N(v)$. For a graph $G$ and a vertex $v \in V(G)$, we write $G-v$ for the graph obtained from $G$ by deleting~$v$. We write $K_t$ for the complete graph with $t$ vertices.

\subsection*{Haussler's Shallow Packing Lemma}
Let $\mathcal{F}$ be a collection of subsets of a finite ground set $V$. We always consider $\mathcal{F}$ as a subset of $2^V$; we do not allow multisets. The \emph{shatter function} of $\mathcal{F}$, denoted by $\pi_{\mathcal{F}}(m)$, is the maximum size of $\mathcal{F}$ when restricted to any $m$ elements in $V$. That is, $\pi_{\mathcal{F}}(m)$ is the maximum, over all $m$-element subsets $W\subseteq V$, of the number of equivalence classes of the relationship~$\sim_{W}$ on $\mathcal{F}$ where two elements $F,F' \in \mathcal{F}$ satisfy $F\sim_{W}F'$ if $F \cap W = F'\cap W$. For a positive integer~$\delta$, we say that two sets $F,F' \in \mathcal{F}$ are \emph{$\delta$-separated} if their symmetric difference $F \Delta F'$ has size at least $\delta$ (that is, there are at least $\delta$ elements in $V$ which are in one of $F,F'$ but not the other). We say that $\mathcal{F}$ is \emph{$\delta$-separated} if any pair of distinct elements in $\mathcal{F}$ are $\delta$-separated.

The following version of Haussler's Shallow Packing Lemma from~\cite{Haussler95} is stated as~\cite[Lemma~2.2]{FPS19}, for instance.  The original statement contains a parameter $d$, which for our purposes can be substituted by $1$, yielding the following.%We just require the case that $d=1$.We just require the case that $d=1$.

\begin{lemma}[\cite{Haussler95}]
\label{lem:packing}
For any number $r\geq 1$, there exists an integer $c=c(r)$ so that for every positive integer $\delta$, if $\mathcal{F}$ is a collection of subsets of a finite ground set $V$ so that $\mathcal{F}$ is $\delta$-separated and $\pi_\mathcal{F}(m)\leq r m$ for every positive integer $m$, then $|\mathcal{F}| \leq c|V|/\delta$.
\end{lemma}

\subsection*{Fundamental Graphs}
Let $M$ be a matroid. Given a basis $B$ of $M$, the \emph{fundamental circuit} of an element $e\in E(M) \setminus B$ with respect to $B$ is the unique circuit of $M$ which is contained in $B \cup \{e\}$.
The \emph{fundamental graph} $\mathcal{F}(M,B)$ is the bipartite graph with sides $B$ and $E(M)\setminus B$ where each element $e \in E(M)\setminus B$ is adjacent to the other elements in its fundamental circuit. If $T$ is a spanning forest of $G$, then the \emph{fundamental graph} of $G$ with respect to $T$ is the graph $\mathcal{F}(M(G), E(T))$.
For convenience, we denote this graph by $\mathcal{F}(G, T)$. 

%\MH{Doesn't $T$ have to be a maximum spanning forest?}
%\RM{Yes, good point, I think there used to be a definition in some other version. I added a definition and included this in the definition.}

Bouchet~\cite{graphicIsoSystems} introduced fundamental graphs for binary matroids and showed how they behave with respect to taking matroid minors. However, since Bouchet used different (more general) terminology, we refer the reader to Oum~\cite{RWAndVM} for a nice introduction to fundamental graphs of binary matroids. For general matroids, see for instance~\cite{GHMO}. The following is well-known and easy to see.

\begin{lemma}%[{\cite[Proposition 3.1\textit{(ii)}]{GHMO}}]
\label{lem:dualFundGraphs}
    For any matroid $M$ and basis $B$ of $M$, we have $\mathcal{F}(M,B)=\mathcal{F}(M^*,E(M)\setminus B)$.%MH{shouldn't it be 'isomorphic'?}\RMM{They're really equal; note that they both have the same vertex-set. Maybe this would read less weirdly with just $=$?}
\end{lemma}

%For a matroid $M$ and element $e \in E(M)$, we write $M\setminus e$ for the matroid obtained from $M$ by deleting $e$, i.e., its bases are the largest independent sets of $M$ in $E\setminus \{e\}$. Furthermore, denote by $M/e$ the matroid obtained from $M$ by contracting $e$, i.e, $(M^*\setminus e)^*$. For a graph $G$ and a vertex $v \in V(G)$, we write $G-v$ for the graph obtained from $G$ by deleting~$v$. 
Consider deleting a vertex from the non-basis side of a fundamental graph; this corresponds to just deleting that element from the matroid. Thus, we obtain the following lemma, using \cref{lem:dualFundGraphs} and the fact that for any matroid $M$ and element $e$, we have $M^*\setminus e = (M/e)^*$.

\begin{lemma}
\label{lem:minorsFundGraphs}
    For any matroid $M$, basis $B$ of $M$, and element $e \in E(M)$,\begin{align*}
    \mathcal{F}(M,B)-e = \begin{cases}
    \mathcal{F}(M\setminus e,B)         &\textrm{if }e \in E(M)\setminus B \cr
    \mathcal{F}(M/ e,B\setminus \{e\})  &\textrm{if }e \in B.  \end{cases}
    \end{align*}
\end{lemma}

In \cref{sec:applications}, we use the Growth Rate Theorem of Geelen and Whittle~\cite{geelen2003cliques} to prove that the fundamental graphs of certain matroids have linear neighborhood complexity (see~\cref{lem:growthRateNeighborhoodCompl}). %In the context of matroids with a forbidden graphic minor, the Growth Rate Theorem says the following.
The Growth Rate Theorem of Geelen and Whittle~\cite{geelen2003cliques} says the following. (We remark that there is also a more general Growth Rate Theorem of Geelen, Kung, and Whittle~\cite{GrowthRateTheorem}.)

\begin{theorem}[\cite{geelen2003cliques}]
\label{thm:growthRate}
    For any integers $t$ and $q$, there exists an integer $g(t,q)$ such that any simple rank-$n$ matroid with no $U_{2,q+2}$ or $M(K_t)$ minor has at most $g(t,q)\cdot n$ elements.
\end{theorem}

\noindent Note that graphic matroids exclude the line $U_{2, 4}$ as a minor. 
We remark that Nelson, Norin, and Omana \cite{nelson2023density} have recently improved the bound on $g(t,q)$ in \cref{thm:growthRate} to a single exponential function.
%\Cref{thm:growthRate} generalizes a classic lemma due to Mader~\cite{MaderAverageDegree} (and optimized by Thomason~\cite{Thomason84} and Kostochka~\cite{Kostochka82, Kostochka84}) about the density of graphs with no $K_t$-minor.

%ODD CHI FUNDAMENTAL
\section{Odd coloring and neighborhood complexity}
\label{sec:oddChiFund}

In this section we prove \cref{thm:main}, which says that graphs of linear neighborhood complexity have bounded improper odd chromatic number. 

First we need to prove a lemma which says that graphs of linear neighborhood complexity contain vertices which have almost the same neighborhood. Let $G$ be a graph, and let $\delta$ be an integer. We say that two vertices $u,v \in V(G)$ are \emph{$\delta$-near-twins} if $|N(u) \Delta N(v)|\leq \delta$, where we write $\Delta$ for the symmetric difference. Likewise, we say that two vertices are \emph{twins} if they have the exact same neighborhood.

We now observe that linear neighborhood complexity is enough to obtain a pair of $\mathcal{O}(1)$-near-twins within the desired side of a bipartite graph (under some mild additional assumptions). In fact, we show that we can force the near-twins to be within any linearly-sized subset of that side of the bipartite graph.

\begin{lemma}
\label{lem:findingNearTwins2}
    For any integers $r$ and $\ell$, there exists an integer $\delta = \delta(r, \ell)$ so that if $G$ is a bipartite graph with a bipartition $(X,Y)$ so that $X$ does not contain any pair of twins and $\eta_G(m) \leq r \cdot m$ for every positive integer $m$, then for every $Y' \subseteq Y$ with $|Y'| \geq \max\{2, |Y|/\ell\}$, there exist distinct vertices $u,v \in Y'$ which are $\delta$-near-twins in $G$.
\end{lemma}
\begin{proof}
Let $c = c(r)$ be the constant from Haussler's Shallow Packing Lemma, which is stated as \cref{lem:packing}. Then set $\delta \coloneqq r \ell c+1$. 

Going for a contradiction, suppose there are no two vertices in $Y'$ which are $\delta$-near-twins in $G$. Let $\mathcal{F}=\{N(v):v \in Y'\}$, so that $\mathcal{F}$ is a collection of subsets of $X$. Since $Y'$ does not contain any pair of twins, $|\mathcal{F}|=|Y'|$. Moreover, since $Y'$ does not contain any pair of $\delta$-near-twins, $\mathcal{F}$ is $\delta$-separated. Finally, by the assumption about the neighborhood complexity of $G$, we have that $\pi_\mathcal{F}(m) \leq rm$ for every positive $m$. So $|Y'| \leq c |X|/\delta$ by \cref{lem:packing}. However, since $X$ contains no twins, we also have\begin{align*}
|X| \leq |\{N(v) \cap Y:v \in V(G)\}|\leq r |Y|\leq r\ell |Y'|\leq r\ell c|X|/\delta.
\end{align*}Thus $\delta \leq r\ell c$, however we chose $\delta = r\ell c+1$, a contradiction.
\end{proof}

We are ready to prove a one-sided odd-coloring result for bipartite graphs, which will be the main ingredient to the full theorem.

\begin{lemma}\label{lem:main}
    For any integer $r$, there exists an integer $k = k(r)$ so that if $G$ is a bipartite graph with bipartition $(X,Y)$ and  neighborhood complexity $\eta_G(m) \leq r \cdot m$ for every positive $m$, then $Y$ can be $k$-colored such that for each non-isolated vertex $v\in X$, there exists a color which appears an odd number of times in $N(v)$.
\end{lemma}
\begin{proof}

Let $\delta$ be the function from \cref{lem:findingNearTwins2}, and set $k \coloneqq \max \left(2r+1, \delta(r, 2r+1) +1\right)$. We now show that a coloring $\varphi_Y$ as claimed exists using at most $k$ colors. Going for a contradiction, suppose otherwise. Let $G$ be a graph with a bipartition $(X,Y)$ which forms a counterexample. Choose $G,X,Y$ so that $|V(G)|$ is minimum. First, we prove two claims.

%\RMM{The claim numbering is a little weird, but maybe that's unavoidable.}
%\MH{I changed it so claims are numbered separately, but still uniquely in the entire document. Is that better?}
%\RMM{Much better!}

\begin{claim}
\label{clm:noXTwins}
There is no pair of twins in $X$.
\end{claim}
\begin{claimproof}
If $X$ contains two vertices $u$ and $v$ which are twins, then by minimality there exists a coloring $\varphi_Y:Y \rightarrow \mathbb{N}$ so that every non-isolated vertex in $X\setminus \{v\}$ has a color that appears an odd number of times in its neighborhood. This same coloring works for~$G$.
\end{claimproof}

Our goal is to apply \cref{lem:findingNearTwins2} to find a pair of near-twins in $Y$. Then we will use the same color for those two vertices, and apply the fact that $G$ is a minimum counterexample. However, we will have to avoid pairs of vertices $u,v \in Y$ so that there exists a degree-$2$ vertex in $X$ whose neighbors are $u$ and $v$. (Such a vertex forces us to use different colors for $u$ and $v$.) The next claim gives us a large subset of $Y$ to work with.

\begin{claim}
\label{clm:largeSubsetOfY}
There exists a set $Y' \subseteq Y$ of size at least $|Y|/(2r+1)$ so that no degree-$2$ vertex in $X$ has both its neighbors in $Y'$.
\end{claim}
\begin{claimproof}
Consider an auxiliary graph $H$ with vertex-set $Y$ where two vertices $u,v \in Y$ are adjacent if there exists a vertex in $X$ whose neighbors are precisely $u$ and $v$. We claim that $H$ is $2r$-degenerate. This will complete the proof since then we can find an independent set of size at least $|Y|/(2r+1)$ by taking a greedy coloring and choosing the largest color class. So, consider an induced subgraph $H'$ of $H$. The number of edges of $H'$ is at most the number of different neighborhoods $G$ has within $V(H')$. So $|E(H')|\leq r\cdot |V(H')|$ by our bound on the neighborhood complexity of $G$. Thus $H$ is $2r$-degenerate, as desired.
\end{claimproof}

%\MH{shouldn't it be $|E(H')|\leq r\cdot |V(H')|$?}
%\RM{Yes, fixed.}

Let $Y' \subseteq Y$ be as in \cref{clm:largeSubsetOfY}. Notice that $|Y'| \geq 2$ since otherwise $|Y| \leq 2r+1$ and we could color every vertex in $Y$ with a different color. Now, notice that $G, X, Y, Y'$ satisfy the conditions of \cref{lem:findingNearTwins2} with $\ell \coloneqq 2r+1$. (Recall that $X$ contains no pair of twins by \cref{clm:noXTwins}.) Thus, there exist distinct vertices $u,v \in Y'$ which are $(k-1)$-near-twins. By the minimality of $G$, there exists a coloring $\varphi:Y\setminus \{u,v\}\rightarrow \mathbb{N}$ with at most $k$ colors so that every vertex in $X$ with at least one neighbor in $Y\setminus \{u,v\}$ has a color that appears an odd number of times in its neighborhood. By the choice of $Y'$, this is actually every vertex in $X$ except for possibly some vertices of degree $0$ or $1$. We do not need to worry about the vertices in $X$ of degree $0$ or $1$ as they automatically satisfy the condition.

Let us say that a color $i$ is \emph{bad} if there exists a vertex $x \in N(u) \Delta N(v)$ so that $i$ is the unique color that appears an odd number of times in $N(x) \cap (Y \setminus \{u,v\})$. Thus, there are at most $|N(u) \Delta N(v)| \leq k-1$ bad colors. It follows that there exists a coloring $\varphi_Y$ of $Y$ which is obtained from $\varphi$ by coloring $u$ and $v$ the same color, which is not bad. %It follows that every non-isolated vertex $x \in X$ has a color that appears an odd number of times in $c_Y$. 
It follows that for every non-isolated vertex $x \in X$ there is a color that appears an odd number of times in $N(x)$. (Note that if $x$ is adjacent to both of $u$ and $v$ or neither of $u$ and $v$, then the number of times each color appears in $N(x)$ has the same parity as in $N(x) \setminus \{u,v\}$.)
\end{proof}

Now we are ready to prove \cref{thm:main}. We restate the theorem below for convenience.

\nghComplex*
\begin{proof}
Denote the tensor product $H=G\times K_2$, with vertex set $V(H)=V(G)\times[2]$ and an edge $\{(u,i),(v,j)\}$ if and only if $\{u,v\}\in E(G)$ and $\{i,j\}=[2]$. If $A\subseteq V(G)\times[2]$ denote $A_i= \{v\in V(G)\mid (v,i)\in A\}$ for $i=1,2$. Then for any $A \subseteq V(G)\times[2]$, $$|\{N(v,i) \cap A: (v,i) \in V(H)\}|\leq |\{N(v) \cap A_1 : v \in V(G)\}|+|\{N(v) \cap A_2 : v \in V(G)\}|.$$
If $|A|=m$, then the latter is at most $2\eta_G(m)$. Hence $\eta_H(m)\leq 2rm$ for all $m$. Further, $H$ is bipartite with bipartition $(X,Y)=(V(G)\times\{1\},V(G)\times\{2\})$. Thus,  we can color $Y$ with $k(2r)$ colors so that every non-isolated vertex in $X$ has a color that appears an odd number of times in its neighborhood, where $k$ is the function from \cref{lem:main} . 
Since $Y=V(G)\times\{2\}$, this coloring can be interpreted as an improper odd coloring of $G$.
\end{proof}

\section{Applications}
\label{sec:applications}
In this section we prove the two corollaries of \cref{thm:main} (\cref{cor:NeighborhoodComplOddColoring} and \cref{cor:oddColoringMat}). 

We begin by using the Growth Rate Theorem (stated as \cref{thm:growthRate}) to control the neighborhood complexity of certain fundamental graphs. We note that we could do a ``one-sided'' version of the following lemma. However this version, which is symmetric under taking duals, is more convenient for our purposes. Note that \cref{cor:oddColoringMat} immediately follows from \cref{thm:main} and the following lemma.

%\JD{The same proof should work for any matroid that is $\GF(q)$-representable or forbids $U_{2,n}$ and $U_{n-2,n}$ as a minor. However we would also need the general analogues of \cref{lem:dualFundGraphs} and \cref{lem:minorsFundGraphs}.}

\begin{lemma}
\label{lem:growthRateNeighborhoodCompl}
For any integers $t$ and $q$, there exists an integer $r=r(t,q)$ so that if $M$ is a matroid with no $U_{2,q+2}$, $U_{q,q+2}$, $M(K_t)$, or $M(K_t)^*$ minor, and $B$ is a basis of $M$, then the neighborhood complexity $\eta_{\mathcal{F}(M,B)}(m)$ of $\mathcal{F}(M,B)$ is at most $r\cdot m$ for every positive $m$. 
\end{lemma}
\begin{proof}
Let $g$ denote the function from \cref{thm:growthRate}; so any simple rank-$n$ matroid with no $U_{2,q+2}$ or $M(K_t)$ minor has at most $g(t,q) \cdot n$ elements. Let $M$ be a matroid with no $U_{2,q+2}$, $U_{q,q+2}$, $M(K_t)$, or $M(K_t)^*$ minor, and let $B$ be a basis of $M$. 

Since $\mathcal{F}(M,B)$ is bipartite, and a matroid and its dual have the same fundamental graph by \cref{lem:dualFundGraphs}, it suffices to consider the number of different neighborhoods within a set $A \subseteq B$. Let $F \subseteq E(M) \setminus B$ be obtained by selecting one vertex $v$ with $N(v) \cap A = A'$ for each subset $A' \subseteq A$ for which $|A'|\geq 2$, and such a vertex $v$ exists. Thus $|\{N(v) \cap A: v \in V(\mathcal{F}(M,B))\}| \leq |F|+|A|+1$ because $A$ has $|A|+1$ subsets with at most one element.

Now, let $M'$ be the matroid obtained from $M$ by contracting all elements in $B\setminus A$ and deleting all elements in $(E(M) \setminus B)\setminus F$. By \cref{lem:minorsFundGraphs}, the subgraph of $\mathcal{F}(M,B)$ induced by $A \cup F$ is equal to $\mathcal{F}(M', A)$. Note that $M'$ is a rank-$|A|$ matroid with no $U_{2,q+2}$ or $M(K_t)$ minor. 

We claim that $M'$ is simple as well. %$|F| \leq g(t)\cdot |A|-|A|$, \KK{Isn't the neighborhood complexity bounded by $|F|+|A|+1$, which is also the quantity bounded by GRT, so why not write $|F|+|A|+1\leq g(t)\cdot |A|$?}.\RMM{I took out the plus one because there's no strict inequality in the statement we wrote of the GRT.} 
First of all, $M'$ has no loops since $A$ is a basis and every vertex in $F$ has at least one neighbor in $A$. Next, $M'$ has no parallel pair with one element in $A$ and the other in $F$ since every element in $F$ has at least two neighbors in $A$. Finally, $M'$ has no parallel pair with both elements in $F$ since no two vertices in $F$ have the same neighborhood in $A$. So $M'$ is simple, and the Growth Rate Theorem implies that $|F|+|A|=|E(M')|\leq g(t,q)\cdot |A|$, which completes the proof.
\end{proof}

Recall that a class of graphs is \emph{$\chio$-bounded} if there exists a function $f$ so that any graph in the class with clique number $\omega$ has odd chromatic number at most $f(\omega)$. Similarly, a class is \emph{$\chi$-bounded} if there exists a function $g$ so that any graph in the class with clique number $\omega$ has chromatic number at most $g(\omega)$. We denote the chromatic number of a graph $G$ by~$\chi(G)$. We say that a class is \emph{hereditary} if it is closed under taking induced subgraphs.

The following lemma from~\cite{jiménez2024boundednessproperconflictfreeodd} implies that in order to understand whether a hereditary $\chi$-bounded class is also $\chio$-bounded, it suffices to consider the bipartite graphs in the class. We state a slightly weaker but simpler version of the lemma. 

\begin{lemma}[{\cite[Lemma~1.3]{jiménez2024boundednessproperconflictfreeodd}}]
\label{lem:reducesToBip}
If every induced bipartite subgraph of a graph $G$ has odd chromatic number at most $t$, then $\chio(G)\leq t^2 \chi(G)$.
\end{lemma}

Now we may prove \cref{cor:NeighborhoodComplOddColoring}. We omit the definitions of merge-width and vertex-minors; instead we refer the reader to~\cite{mergeWidthIntro} and~\cite{Davies2022vm}, respectively. The corollary is restated below for convenience.

\corWidthEtc*

\begin{proof}
First suppose that $\mathcal{G}$ has bounded merge-width. Bonamy and Geniet~\cite{mergewidthChi} recently proved that classes with bounded merge-width are $\chi$-bounded and have linear neighborhood complexity. Thus, by \cref{rem:product} and \cref{thm:main}, the class $\mathcal{G}$ is $\chio$-bounded.

%and classes with a forbidden vertex-minor~\cite{Davies2022vm}

Now suppose that $\mathcal{G}$ has a forbidden vertex-minor. We may assume that $\mathcal{G}$ is hereditary because taking induced subgraphs cannot create a new vertex-minor. Moreover, $\mathcal{G}$ is $\chi$-bounded by~\cite{Davies2022vm}. Thus, by \cref{lem:reducesToBip}, it suffices to prove that the bipartite graphs in $\mathcal{G}$ have bounded odd chromatic number. 

We observe that the $1$-subdivision of $K_t$ contains every $t$-vertex graph as a vertex-minor. Since every pivot-minor is also a vertex-minor, it follows that any class of bipartite graphs with a forbidden vertex-minor also forbids the $1$-subdivision of $K_t$ as a pivot-minor (for some $t$). Now, consider a bipartite graph $G \in \mathcal{G}$, and let $A$ denote the bipartite adjacency matrix of $G$. (So the rows of $A$ correspond to one side of the bipartition, the columns of $A$ correspond to the other, and each entry is $0$ or $1$ depending on whether the corresponding vertices are adjacent.) Let $M$ denote the column matroid of the binary matrix $[I \text{ | } A]$ where $I$ is an identity matrix whose size equals the number of rows of $A$. Then $G$ is the fundamental graph of $M$ with respect to the basis corresponding to the columns of $I$. Also notice that the $1$-subdivision of $K_t$ is itself the fundamental graph of $M(K_{t+1})$ with respect to the basis formed by the edges of a $(t+1)$-vertex star. 

Bouchet~\cite{graphicIsoSystems} proved a connection between minors of binary matroids and pivot-minors of their fundamental graphs which yields the following result; if $M(K_{t+1})$ or its dual is a minor of $M$, then the $1$-subdivision of $K_t$ is a pivot-minor of $G$. It follows that $M$ forbids both $M(K_{t+1})$ and its dual as minors. Since $M$ is binary, it also forbids $U_{2,4}$ as a minor. So by \cref{lem:growthRateNeighborhoodCompl}, the bipartite graphs in $\mathcal{G}$ have linear neighborhood complexity. Thus, by \cref{thm:main}, the bipartite graphs in $\mathcal{G}$ have bounded improper odd chromatic number. This also bounds their odd chromatic number by \cref{rem:product}, as desired.
\end{proof}

%BIPARTITE CIRCLE GRAPHS
\section{Odd coloring fundamental graphs of graphs excluding a minor}
\label{sec:bipartiteCircle}

In the following we consider an asymmetric situation, in which we want to color only $E\setminus B$ in $\mathcal{F}(M,B)$ such that every non-isolated vertex from $B$ has some color occurring an odd number of times in its neighborhood. In fact, we only consider this setting in graphic matroids of graphs with no $K_t$-minor.
That is, we give another proof that for any $t$ there is a constant $k = k(t)$ so that given any graph $G$ with no $K_t$-minor and any spanning forest $T$ of $G$, we can color the non-forest edges with $k$ colors so that every non-bridge edge $e$ has a color which appears an odd number of times in its neighborhood in $\mathcal{F}(G,T)$.

Given a graph $G$ and a spanning forest $T$ of $G$, we denote the \emph{fundamental cut} of $T$ with respect to an edge $f \in E(T)$ by $C^*(T,f)$, i.e., it contains all edges in $E(G)$ which join the vertex-sets of the two components of $T-f$ which are not components of $T$. Note that $f$ itself is included in $C^*(T,f)$. Indeed, $C^*(T,f)\setminus\{f\}$ is the neighborhood of $f$ in $\mathcal{F}(G,T)$.

%\MH{As $T$ is a forest $T-f$ might have way more than two components. Should we be more specific saying the two components of $T-f$ that aren't components of $T$ as well?}
%\RMM{Corrected.}

\begin{lemma}
\label{lem:one-fundCut}
    For any graph $G$, any spanning forest $T$ of $G$, and any edge $f_0$ of $T$, there exists a coloring $\varphi: C^*(T,f_0) \setminus \{f_0\} \rightarrow \{1,\ldots,9\}$ so that for each forest-edge $f \in E(T)$ with $C^*(T,f) \cap (C^*(T,f_0) \setminus \{f_0\}) \neq \emptyset$ there exists a color $i \in \{1,\ldots,8\}$ which occurs an odd number of times on $C^*(T,f) \cap (C^*(T,f_0) \setminus \{f_0\})$.
\end{lemma}
\begin{proof}
    % Let $T'$ and $T''$ be the two subtrees of $T$ with $T' \cap T'' = \{t_0\}$ and $T' \cup T'' = T$.
    % Let $t = v'v''$ and consider the component $T'$ of $T - t_0$ rooted at $v'$, as well as the component $T''$ of $T - t_0$ rooted at $v''$.
    The claimed coloring can then be defined independently on each connected component of $G$. Hence, we assume that $G$ is connected and $T$ is a spanning tree. 
    Let $T_1$ and $T_2$ be the two components of $T - f_0$.
    We shall color the non-tree edges in the fundamental cut $C^*(T,f_0)$ of $f_0$, i.e., the non-tree edges with one endpoint in $T_1$ and the other endpoint in $T_2$.
    Consider the tree $T_1' = T_1 + f_0$ to be rooted at the endpoint of $f_0$ in $T_2$.
    Observe that for any edge $f$ in $T_1'$, a non-tree edge $e \in C^*(T,f_0)$ lies in $C^*(T,f)$ if and only if the endpoint of $e$ in $T_1$ lies in the subtree of $T_1'$ below $f$.
    It follows that the set family $\mathcal{S}_1 = \{ C^*(T,f) \cap (C^*(T,f_0) \setminus \{f_0\}) \colon f \in T_1'\}$ is \emph{laminar}, i.e., its members are either disjoint or in a containment relation.
    From~\cite[Lemma 3.8]{jiménez2024boundednessproperconflictfreeodd} we get that there is a coloring $\varphi_1: C^*(T,f_0) \setminus \{f_0\} \to\{1,2,3\}$ such that every element of $\mathcal{S}_1$ has an odd number of elements colored $1$ or an odd number of elements colored~$2$.

    Symmetrically, there is a coloring $\varphi_2: C^*(T,f_0) \setminus \{f_0\} \to\{1,2,3\}$ such that every element of $\mathcal{S}_2 = \{ C^*(T,f) \cap (C^*(T,f_0) \setminus \{f_0\}) \colon f \in T_2'\}$ has an odd number of elements colored $1$ or an odd number of elements colored~$2$.
    Setting $\varphi(e) = (\varphi_1(e),\varphi_2(e))$ for every $e \in C^*(T,f_0) \setminus \{f_0\}$, we obtain the desired $9$-coloring.
    Indeed, if $C^*(T,f) \cap (C^*(T,f_0) \setminus \{f_0\})$ has an odd number of elements colored $i$ in $\varphi_1$, then in $\varphi$ these elements are distributed into three color classes $(i,1)$, $(i,2)$, $(i,3)$, at least one of which is necessarily odd.
    {Note that since by~\cite[Lemma 3.8]{jiménez2024boundednessproperconflictfreeodd} in both factors of the product coloring the color $3$ is not needed in order to witness having an odd color in the neighborhoods, we can designate color $(3,3)$ as color $9$ from the statement.}
\end{proof}

%For the next result, recall that the \emph{star-arboricity} $\mathrm{sa}(G)$ of a graph $G$ is the smallest $\ell$, such that $E(G)$ can be decomposed into $\ell$ star forests, i.e., each component has diameter at most $2$. Note that the $\mathrm{sa}(G)$ is bounded by twice the degeneracy of $G$, see~\cite{AA89}. Hence, in particular for any integer $t$, there exists an integer  Indeed, by the Kostochka–Thomason bound $\ell(t)\in O(t\sqrt{\log t})$, see~\cite{Kos82,Kos86,Tho01,Tho84}.  
We are ready to prove the main theorem of the section, where we recall that $\ell(t)$ is the smallest integer so that if $G$ is a graph with no $K_t$ minor, then the star-arboricity of $G$ is at most $\ell(t)$.

\thmfundCuts*

% \begin{theorem}
% \label{thm:fundCuts}
%     % For any graph $G$ and any spanning forest $T$ of $G$, there exists a function $\phi:E(G)\setminus E(T) \rightarrow \{1,2,\ldots,25\}$ so that for each non-bridge edge $e \in E(T)$, there exists a color $i \in \{1,2,\ldots, 25\}$ which occurs an odd number of times on $C^*(T,e)$.
%     %%% new version %%%
%     For any integer $t$, there exists an integer $k = k(t)$ so that if $G$ is a graph with no $K_t$ minor, and $T$ is a spanning forest of $G$, then there exists a coloring $\phi:E(G)\setminus E(T) \rightarrow \{1,\ldots,k\}$ so that for each non-bridge edge $f \in E(T)$, there exists a color $i \in \{1,\ldots,k\}$ which occurs an odd number of times on $C^*(T,f)$. Indeed, $k(t)\leq 16\ell(t-1)+1$, with $\ell(\cdot)$ defined as above the theorem.
% \end{theorem}
% 
\begin{proof}
    Note that we may assume that $G$ is connected and $T$ is a spanning tree by considering the components separately.
    Further, we may assume that $G$ is $2$-edge-connected by contracting all bridges.

    We root $T$ at an arbitrary vertex $r$.
    We shall construct a sequence $G_0 \subseteq G_1 \subseteq \cdots \subseteq G_x = G$ of subgraphs of $G$, starting with $G_0$ being just vertex $r$.
    For a non-tree edge $e$, the \emph{fundamental cycle} $C(T,e)$ is the unique cycle in $T + e$ and for a tree-edge $f$ denote by $C^*(T,f)$ the fundamental cut of $T$ with respect to $f$.
    Note that for a tree-edge $f$ and a non-tree edge $e$ we have $f \in C(T,e)$ if and only if $e \in C^*(T,f)$.
    For $i > 0$, having defined $G_{i-1}$, let
    \[
        G_i \coloneqq \bigcup \{C(T,e) \colon e \in E(G)\setminus E(T), C(T,e) \text{ contains a vertex of } G_{i-1}\}.
    \]
    Observe that for each $i\geq 0$ we have that $T_i = T \cap G_i$ is a subtree of $T$ containing the root $r$.
    For convenience, for each $i > 0$ we set $A_i \coloneqq E(T_i) - E(T_{i-1})$ as the set of tree-edges added in step~$i$, and $B_i \coloneqq (E(G_i) - E(G_{i-1})) - E(T)$ as the set of non-tree edges added in step~$i$.

    \begin{claim}
    \label{claim:tree-nontree}
        If for a tree-edge $f \in A_i$ we have $C^*(T,f) \cap B_j \neq \emptyset$, then $i\leq j \leq i+1$.
    \end{claim}
    \begin{claimproof}
        Consider any non-tree edge $e \in B_j$ that lies on the fundamental cut $C^*(T,f)$ of the tree-edge $f \in A_i$.
        Then, $f$ lies on the fundamental cycle $C(T,e)$, and it follows that $j \leq i+1$.
        
        Conversely, since $e \in B_j$, the fundamental cycle $C(T,e) \subseteq G_j$ and thus $i \leq j$.
    \end{claimproof}

    Next, we color the edges in $B_i$ (which is non-empty only for $i > 0$), i.e., the non-tree edges in $G_i$ that are not in $G_{i-1}$.
    To this end, consider $G_i$ with spanning tree $T_i$.
    Contract the subtree $T_{i-1}$ of $T_i$ to a single vertex $a$.
    Also, consider the components of $T_i$ after the removal of all vertices in $T_{i-1}$ and contract each such component to a single vertex.
    The resulting graph $H$ is a minor of $G_i$ with two vertices in $H-a$ being adjacent if and only if some non-tree edge in $B_i$ connects the corresponding subtrees of $T_i - V(T_{i-1})$.
    
    Since $a$ is a universal vertex in $H$, the graph $H - a$ has no $K_{t-1}$ minor.
    Thus, as argued above the theorem there exists a universal constant $\ell = \ell(t-1)$ such that $E(H-a)$ can be decomposed into $\ell$ star forests $F_1,\ldots,F_\ell$.
    (We may assume that each vertex of $H-a$ is the center of at least one star.)
    For each star $S$ in $F_1,\ldots,F_\ell$, consider its center $v$ in $H-a$, the corresponding component $K$ of $T_i - V(T_{i-1})$, and the edge $t_0 \in T_i$ that connects $K$ and $T_{i-1}$.
    Observe that $C^*(T,f_0) \setminus \{t_0\} \subseteq B_i$.
    Let $B(S) \subseteq B_i$ be the subset of non-tree edges that correspond in $H$ to edges in the star $S$ or the edge between $a$ and $v$.
    By \cref{lem:one-fundCut} there is a $9$-coloring $\varphi_S$ of $B(S)$ such that for every $f \in T$ the set $C^*(T,f) \cap (C^*(T,f_0) \setminus \{f_0\}) \cap B(S)$ is either empty or contains a color an odd number of times.

    To color $B_i$, we use a new palette of $9$ colors for each of the $\ell$ star forests $F_1,\ldots,F_\ell$, while reusing the same set $C_x$ of $9$ colors for all stars in the same star forest $F_x$.
    In fact, if $S_1$ and $S_2$ are two stars from the same forest $F_x$, then for every tree-edge $f \in A_i$ we have that $C^*(T,f)$ is either disjoint from $B(S_1)$ or $B(S_2)$ (or both).
    Hence, for every $f \in A_i$ we have that $C^*(T,f) \cap \left( \bigcup_{S \in F_x} B(S) \right)$ is either empty or contains a color in $C_x$ an odd number of times.
    
    Finally, we use one set of $\ell$ palettes of $9$ colors to color $B_i$ for even $i$, and a second set of $\ell$ palettes of $9$ colors to color $B_i$ for odd $i$.
    By \cref{claim:tree-nontree}, this results in a coloring $\phi$ with $2\cdot9\cdot \ell$ colors that is odd for $\{C^*(T,f) \colon f \in T\}$. {Using further, that by \cref{lem:one-fundCut} in none of the palettes the color $9$ is used as the odd color in a neighborhood, we can merge all these, yielding a total of $2\cdot8\cdot \ell+1$ colors.}
\end{proof}

We now use \cref{thm:fundCuts} to prove \cref{cor:coloringFundPlanar}, which is restated below. 

\colBipCircle*
\begin{proof}
    De Fraysseix~\cite{deFraysseix81} proved that every connected bipartite circle graph is a fundamental graph of a planar graph $G$ and a spanning tree $T$.
    So $G$ excludes $K_5$ as a minor, and the graphs excluding $K_4$ as a minor are exactly the graphs of treewidth at most $2$. Hence, by~\cite{DOSV98} they have star-arboricity at most $3$.
    Thus, we can use \cref{thm:fundCuts} to color $E(G)\setminus E(T)$ with $16 \cdot 3 + 1 = 49$ colors so that for each non-bridge edge $f \in E(T)$, there exists a color $i \in \{1,2,\ldots,49\}$ which occurs an odd number of times on the fundamental cut of $f$.
    % \MH{Our \cref{thm:fundCuts} does not give a precise bound. If we want to use it to obtain precise bounds, shouldn't we adjust the statement of it?}
    % In fact, the contracted graph $H-a$ in the proof of \cref{thm:fundCuts} is outerplanar and hence can be decomposed into $\ell = 3$ star forests~\cite{HMS96}.

    Now, if $G^*$ denotes the planar dual graph of $G$, the we have $M(G^*)\cong M(G)^*$, see e.g., \cite[Chapter 6.3, Corollary 1]{Wel76}.
    Thus, by \cref{lem:dualFundGraphs} the fundamental graph of $G$ and $T$ is isomorphic to the fundamental of graph $G^*$ with spanning forest $T^*=E(G)\setminus E(T)$.
    Applying \cref{thm:fundCuts} to this graph, we can use a second set of size $49$ to color $E(G^*)\setminus E(T^*)=E(T)$ such that for each non-bridge edge $f \in E(T^*)=E(G)\setminus E(T)$, there exists a color $i \in \{50,51,\ldots, \maxCircle\}$ which occurs an odd number of times on the fundamental cut of $f$ in $G^*$ with respect to $T^*$.
% We have colored both sided of the bipartite circle graph such that every non-isolated vertex has some color an odd number of times in its neighborhood. 
\end{proof}
%%%% OLD VERSION %%%%
% \begin{proof}
% De Fraysseix~\cite{deFraysseix81} proved that every connected bipartite circle graph is a fundamental graph of a planar graph $G$ and a spanning tree $T$. By \cref{thm:fundCuts} we can $25$-color $E(G)\setminus E(T)$ so that for each non-bridge edge $e \in E(T)$, there exists a color $i \in \{1,2,\ldots, 25\}$ which occurs an odd number of times on the fundamental cut of $e$. Now, by planar duality $M(G^*)\cong M(G)^*$, see e.g., \cite[Chapter 6.3, Corollary 1]{Wel76}. Thus, by \cref{lem:dualFundGraphs} the fundamental graph of $G$ and $T$ is isomorphic to the fundamental of graph $G^*$ with spanning forest $T^*=E(G)\setminus E(T)$. Applying to this graph \cref{thm:fundCuts}, we can use a second set of size $25$ to color $E(G^*)\setminus E(T^*)=E(T)$ for each non-bridge edge $e \in E(T^*)=E(G)\setminus E(T)$, there exists a color $i \in \{26,27,\ldots, \maxCircle\}$ which occurs an odd number of times on the fundamental cut of $e$ in $G^*$ with respect to $T^*$.
% We have colored both sided of the bipartite circle graph such that every non-isolated vertex has some color an odd number of times in its neighborhood. 
% \end{proof}

\propImpro*
\begin{proof}
Let $G_n=(V,E)$ be the circle graph given by vertex set $V=\{(a,b)\mid 1\leq a<b\leq n\}$ and edge set $E=\{\{(a,b),(c,d)\}\mid a<c<b<d \text{ cyclically}\}$, see \cref{fig:G6} for an example.

\begin{figure}[ht]
    \centering
    \includegraphics[width=0.3\linewidth]{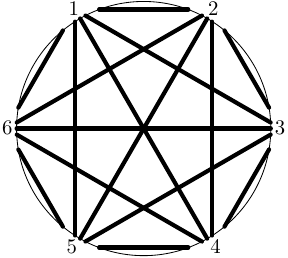}
    \caption{A representation of $G_6$ as circle graph.}
    \label{fig:G6}
\end{figure}

Suppose now that $f:V\to [k]$ is an (improper) odd $k$-coloring and denote by $o_f:V\to 2^{[k]}\setminus\{\emptyset\}$ the function that to every vertex of $G_n$ assigns the (non-empty) set of colors that appear an odd number of times in the neighborhood of the vertex. If $n\geq 2^k$ then there are vertices  $(a,b), (a,c)$ with $o_f(a,b)=o_f(a,c)$. But now, observe that $N(b,c)=N(a,b)\Delta N(a,c)$ and hence $o_f(b,c)=o_f(a,b)\Delta o_f(a,c)=\emptyset$, which is a contradiction.
\end{proof}

We have shown that bipartite circle graphs have odd chromatic number at most \maxCircle. However, the best lower bound we know of is $4$, attained by the cycle of length $4$. It would be interesting to obtain tighter bounds.

\begin{problem}
\label{prob:circle}
Does every bipartite circle graph have odd chromatic number at most $4$? 
\end{problem}

We think that \cref{thm:fundCuts} is interesting in its own right and in fact we are not aware of any family of graphs that forces an unbounded number of colors when coloring non-tree edges as in the theorem. In other words, we wonder if the following holds.

\begin{problem}
\label{prob:primal}
Does there exist a constant $k$ so that if $T$ is a spanning forest of a graph $G$, then there exists a coloring $\phi:E(G)\setminus E(T) \rightarrow \{1,\ldots,k\}$ so that for each non-bridge $f \in E(T)$, there exists a color $i \in \{1,\ldots,k\}$ which occurs an odd number of times on the fundamental cut of $T$ with respect to $f$?
\end{problem}

\section*{Acknowledgments}
This work was completed at the 12th Annual Workshop on Geometry and Graphs held at the Bellairs Research Institute in February 2025. We are grateful to the organizers and participants for providing a excellent research environment. We would also like to thank Szymon Toru\'{n}czyk for helpful comments and Juan Pablo Peña for discussions on circle graphs.

%\RMM{Ug, references that I added from my older bib files have first names abbreviated.}\KK{cahnged style so that all are abreviated}

\bibliographystyle{abbrv}
\bibliography{matroids}

\end{document}